\documentclass [reqno] {amsart}

\usepackage{amsfonts}
\usepackage{amssymb}

\usepackage{color}
\usepackage{hyperref}

\newtheorem{thm}{Theorem}
\newtheorem{lemma}{Lemma}
\theoremstyle{definition}

\begin{document}


\title[GENERALIZED NAVIER-STOKES EQUATIONS]
 {GENERALIZED NAVIER-STOKES EQUATIONS, ASSOCIATED WITH THE DOLBEAULT COMPLEX}

\author[A.~Shlapunov]{Alexander Shlapunov}
\address{Siberian Federal University,
	Institute of Mathematics and Computer Science,
	pr. Svobodnyi 79,
	660041 Krasnoyarsk,
	Russia}
\email{ashlapunov@sfu-kras.ru}  
\author[A.~Polkovnikov]{Alexander Polkovnikov}
\email{paskaattt@yandex.ru}  







\begin{abstract} 
We consider the Cauchy problem in the band $\mathbb{C}^{n}\times[0, T], n>1,T>0$, for a system of nonlinear 
differential equations structurally similar to the classical Navier-Stokes equations for an incompressible fluid. 
The main difference of this system is that it is generated not by the standard gradient operators $\nabla$, 
divergence div and rotor rot, but by the multidimensional Cauchy-Riemann operator $\overline{\partial}$ in 
$\mathbb{C}^{n}$, its formally adjoint operator $\overline{\partial}^{*}$ and the compatibility complex for 
$\overline{\partial}$, which is usually called the Dolbeault complex. The similarity of the structure makes it 
possible to prove for this problem the theorem of the existence of weak solutions and the open mapping theorem on 
the scale of specially constructed Bochner-Sobolev spaces. In addition, a criterion for the 
existence of a ``strong'' solution in these spaces is obtained.
\end{abstract}





\maketitle

\subsection{Introduction}

The Navier-Stokes equations, see, for example, the work \cite{Lad03} or the monographs \cite{Lad61}, \cite{Temam79}, and the bibliography therein, have been a challenge for both theoretical mathematicians and specialists in applied mathematics and hydrodynamics for many decades. In the work \cite{NSE19}, a more general problem was proposed within the framework of the theory of differential complexes, and in the work \cite{OMTNS21}, this problem was considered for the de Rham complex in $\mathbb{R}^{n}, n>1$, on the scale of specially constructed Bochner-Sobolev spaces (in the first degree of the complex, the corresponding system coincides with the Navier-Stokes equations for incompressible fluid). The Dolbeault complex in the complex space $\mathbb{C}^{n}$ shares many features with the de Rham complex, but differs from it in several important aspects, notably the general non-finiteness of the solution space of the operator in the zero degree of the complex and the subellipticity of the corresponding Neumann problems (whereas for the de Rham complex, the solutions in the zero degree of the complex are constant, and the Neumann problems are elliptic).

\subsection{The Generalized Stokes and Navier-Stokes Operators}

Let $\overline{\partial}$ denote the multidimensional Cauchy-Riemann operator in $\mathbb{C}^{n} \cong \mathbb{R}^{2n}$, $n \geq 2$, i.e., a column of differential operators $\left(\overline{\partial}_{1}, \ldots, \overline{\partial}_{n}\right)^{T}$, whose components are one-dimensional Cauchy-Riemann operators $\overline{\partial}_{j}=\frac{1}{2}\left(\frac{\partial}{\partial x_{j}} + \iota \frac{\partial}{\partial x_{j+n}}\right)$, where $\iota$ is the imaginary unit. Similar to the gradient operator in $\mathbb{R}^{n}$, the operator $\overline{\partial}$ naturally generates a complex
\[
0 \rightarrow C_{\Lambda^{0, 0}}^{\infty} \xrightarrow{\overline{\partial}^{0}} C_{\Lambda^{0, 1}}^{\infty} \xrightarrow{\overline{\partial}^{1}} \cdots \xrightarrow{\overline{\partial}^{n-1}} C_{\Lambda^{0, n}}^{\infty} \rightarrow 0,
\]
where ${C}^\infty_{\Lambda^{p, q}}$ is the space of exterior differential forms of bidegree $(p, q)$ with infinitely smooth coefficients with respect to the variables $z = \left(z_{1}, \ldots, z_{n}\right)$, $\overline{z} = \left(\overline{z}_{1}, \ldots, \overline{z}_{n}\right)$, $z_{j} = x_{j} + \iota x_{j+n}$, and $\overline{\partial}^{0} = \overline{\partial}$, $\overline{\partial}^{q+1} \circ \overline{\partial}^{q} = 0$ for $0 \leq q \leq n-1$ (see \cite[\S 1.2]{CDE95}). Denoting by $\left(\overline{\partial}^{q}\right)^{*}$ the formally adjoint operator to $\overline{\partial}^{q}$, we obtain a set of strongly elliptic operators (generalized Laplacians of the Dolbeault complex) $\Delta^{q} = \left(\overline{\partial}^{q}\right)^{*} \overline{\partial}^{q} + \overline{\partial}^{q-1}\left(\overline{\partial}^{q-1}\right)^{*}$ for $0 \leq q \leq n$, where, by default, $\overline{\partial}^{-1} = 0$ and $\overline{\partial}^{n} = 0$.

In the work \cite{NSE19}, a construction of generalized Navier-Stokes equations associated with differential complexes was proposed. With the refinements from \cite{Polk23}, in this context, we consider the following system of equations: given a $(0, q)$-differential form $f$ on $\mathbb{C}^{n} \times[0, T)$ and a $(0, q)$-differential form $u_{0}$ on $\mathbb{C}^{n}$, find a $(0, q)$-differential form $u$ and a $(0, q-1)$ differential form $p$ on $\mathbb{C}^{n} \times[0, T)$ satisfying
\begin{equation}\label{zad}
	\left\{\begin{array}{rcl}
		\partial_{t} u + \mu \Delta^{q} u + \mathcal{N}^{q} u + \overline{\partial}^{q-1} p &=& f \text{ in } \mathbb{C}^{n} \times(0, T), \\
		\left(\overline{\partial}^{q-1}\right)^{*} u &=& 0 \text{ in } \mathbb{C}^{n} \times(0, T), \\
		\left(\overline{\partial}^{q-2}\right)^{*} p &=& 0 \text{ in } \mathbb{C}^{n} \times(0, T), \\
		u(z, 0) &=& u_{0}(z), \, z \in \mathbb{C}^{n}, \\
		\sup\limits_{t \in [0, T]} \int\limits_{\mathbb{C}^{n}} |u(z, t)|^{2} \, dx + \int\limits_{0}^{T} \int\limits_{\mathbb{C}^{n}} \sum\limits_{j=1}^{2n}\left|\partial_{j} u(z, t)\right|^{2} \, dx \, dt &<& +\infty,
	\end{array}\right.
\end{equation}
where $\mu$, $T$ are fixed positive numbers, the coefficients of all forms depend on the parameter $t$, and $\mathcal{N}^{q}$ is an appropriate nonlinear operator. Note that the strong ellipticity of the operator $\Delta^{q}$ means that the operator $\partial_{t} + \mu \Delta^{q}$ is strongly uniformly parabolic in the sense of Petrovski. As for the nonlinearity $\mathcal{N}^{q}$, we will restrict ourselves to the following case. Fix two bilinear differential operators of zero order with constant coefficients:
$$
M_{1}^{q}: C_{\Lambda^{0, q+1}}^{\infty} \times C_{\Lambda^{0, q}}^{\infty} \rightarrow C_{\Lambda^{0, q}}^{\infty}, \quad M_{2}^{q}: C_{\Lambda^{0, q}}^{\infty} \times C_{\Lambda^{0, q}}^{\infty} \rightarrow C_{\Lambda^{0, q-1}}^{\infty}
$$
and set $\mathcal{N}^{q} u = M_{1}^{q}\left(\overline{\partial}^{q} u, u\right) + \overline{\partial}^{q-1} M_{2}^{q}(u, u)$.

Following the classical scheme of studying the Navier-Stokes equations, a theorem on the existence of weak solutions to problem \eqref{zad} can be obtained under additional conditions on the form $M_{1}^{q}$. Specifically, let $L_{\Lambda^{p, q}}^{r}$ denote the space of differential forms

\[
u = \sum_{\# I = p} \sum_{\# J = q} u_{I J} \, dz_{I} \wedge d\overline{z}_{J}
\]
of bidegree $(p, q)$ on $\mathbb{C}^{n}$ with components $u_{I J}$ in $L^{r}\left(\mathbb{C}^{n}\right)$; we endow it with the norm

\[
\|u\|_{L_{\Lambda^{p,q}}^{r}} = \left(\sum_{\# I = p} \sum_{\# J = q} \int_{\mathbb{R}^{2n}} \left|u_{I J}(x)\right|^{r} \, dx\right)^{1 / r}.
\]

Similarly, spaces of forms on $\mathbb{C}^{n}$ with Sobolev class components $W_{\Lambda^{p, q}}^{s, r}$ and $H_{\Lambda^{p, q}}^{s}$ are introduced. For convenience, the specific cases of the introduced spaces for forms of bidegree $(0, q)$ will be denoted respectively as $L_{q}^{r}$, $W_{q}^{s, r}$, and $H_{q}^{s}$. Next, let $\mathcal{V}_{\Lambda^{0, q}}$ denote the subspace in $C_{0, \Lambda^{0, q}}^{\infty}$ consisting of forms that satisfy $\left(\overline{\partial}^{q-1}\right)^{*} u = 0$ in $\mathbb{C}^{n}$, and let $\mathbf{H}_{q}^{s}$ be the closure of $\mathcal{V}_{\Lambda^{0, q}}$ in $H_{q}^{s}$ for $s \in \mathbb{Z}_{+}$. As usual, $\left(\mathbf{H}_{q}^{s}\right)^{\prime}$ denotes the dual space to $\mathbf{H}_{q}^{s}$. Furthermore, if $I = [0, T]$, $p \geq 1$, and $\mathcal{B}$ is a Banach space (of functions on $\mathbb{C}^{n}$), then $L^{r}(I, \mathcal{B})$ will denote the Bochner space of measurable mappings $u: I \rightarrow \mathcal{B}$ with the norm

\[
\|u\|_{L^{r}(I, \mathcal{B})} := \| \|u(\cdot, t)\left\|_{\mathcal{B}}\right\|_{L^{r}(I)}, r \geq 1,
\]

see, for example, \cite[ch. III, \S 1]{Temam79}. Similarly, $C(I, \mathcal{B})$ spaces are introduced, that is, the spaces of all mappings $u: I \rightarrow \mathcal{B}$ with the norm

\[
\|u\|_{C(I, \mathcal{B})} := \sup_{t \in I}\|u(\cdot, t)\|_{\mathcal{B}}.
\]

The following theorem shows that under additional constraints on the nonlinear term $\mathcal{N}^{q} u$, the problem \eqref{zad} has a weak solution.
\begin{thm}\label{slsol}
	Let $s \in \mathbb{N}$ and $s \leq n \leq (s+1)$. If
	\begin{equation}\label{key1}
		\left(M_{1}^{q}\left(\overline{\partial}^{q} w, v\right), v\right)_{L_{ q}^{2}} = 0 \text{ for all } v \in \mathcal{V}_{\Lambda^{0, q}},
	\end{equation}
	then for any pair $\left(f, u_{0}\right) \in L^{2}\left(I, \left(\mathbf{H}_{q}^{1}\right)^{\prime}\right) \times \mathbf{H}_{q}^{0}$, there exists a differential form $u \in L^{\infty}\left(I, \mathbf{H}_{q}^{0}\right) \cap L^{2}\left(I, \mathbf{H}_{q}^{1}\right)$ satisfying
	\begin{equation}\label{zadsl}
		\left\{\begin{array}{rcl}
			\frac{d}{dt}(u, v)_{L^{2}_q} + \mu\left(\overline{\partial}^{q} u, \overline{\partial}^{q} v\right)_{L^{2}_{q+1}} & = & \left\langle f - \mathcal{N}^{q} u, v\right\rangle_{\Lambda^{0, q}} \\
			u(\cdot, 0) & = & u_{0}
		\end{array}\right.
	\end{equation}
	for all $v \in \mathbf{H}_{q}^{s}$. Moreover, $\partial_{t} u \in L^{\frac{2}{n+1-s}}\left(I, \left(\mathbf{H}_{q}^{s}\right)^{\prime}\right)$.
\end{thm}

\begin{proof}\label{key2}
		The proof is conducted similarly to the theorem on the existence of weak solutions for the Navier-Stokes equations, based on energy estimates, Gagliardo-Nirenberg inequalities for ${\mathbb{R}}^{2n}$, and using the Galerkin method; see, for example, \cite{LiMa72}, \cite{Temam79}, or \cite{Lad61}.
		
\end{proof}

A key factor in the proof is the condition \eqref{key1}, which also holds for the nonlinearity arising in the Navier-Stokes equations, allowing us to prove the existence of a weak solution for the latter. As is usual for Navier-Stokes type equations, in this case, it is not possible to prove a uniqueness theorem for weak solutions $u$ of equations \eqref{zad}, i.e., satisfying \eqref{zadsl}, and the unknown form $p$ is identified (additively, up to a form with constant coefficients) only in the space of distributions using information about the cohomology of the Dolbeault complex; see \cite{OMTNS21}, \cite{Temam79}. For example, for the Navier-Stokes equations, it has long been known that the existence of regular solutions follows from the existence of a so-called strong solution, i.e., a weak solution in the Bochner space $L^\mathfrak{s}([0, T], L^\mathfrak{r}({\mathbb R^n}))$ with numbers $\mathfrak{r}$ and $\mathfrak{s}$ satisfying the relation $2/\mathfrak{s} + n/\mathfrak{r} = 1$ for $\mathfrak{r} > n$ (see, for example, \cite{Lad61}, \cite{Pro59}, \cite{Serr62} and the refinement \cite{ESS} for the case $\mathfrak{r} = n = 3$). For the problem \eqref{zad}, we can provide a similar criterion for the existence of a strong solution.

\begin{thm}\label{crstrong}
	Let the conditions of Theorem \eqref{slsol} be satisfied. If the solution to the problem \eqref{zadsl} lies in the space $L^\mathfrak{s}(I, L_{q}^\mathfrak{r})$ with some numbers $\mathfrak{r} > 2n$ and $\mathfrak{s}$ satisfying the relation $2/\mathfrak{s} + 2n/\mathfrak{r} = 1$, then the problem \eqref{zadsl} has a smooth solution given a smooth right-hand side.
\end{thm}
\begin{proof}
	This follows from the Gagliardo-Nirenberg inequalities using the Galerkin method and is quite analogous to the classical case for the Navier-Stokes equations (see, for example, \cite{Lad61}, \cite{Serr62}, \cite{Pro59}).
\end{proof}

Note that for $n=1$, the Cauchy-Riemann operator generates a trivial compatibility complex consisting of only one operator, and therefore the problem \eqref{zadsl} makes no sense. For $n>1$, the real dimension of the space ${\mathbb{C}}^n$ is $2n>2$, which means that the standard Gagliardo-Nirenberg inequalities cannot guarantee the uniqueness theorem for the problem \eqref{zadsl} in this case.

Also note that for $q=1$, there is a natural nonlinearity $\mathcal{N}^{q} u=\overline{\star}\left(\overline{\star}^{1} u \wedge u\right)+\overline{\partial}^{0}|u|^{2}$, structurally corresponding to the Lamb form of nonlinearity included in the Navier-Stokes equations and satisfying \eqref{key1}; here, $\star$ is the Hodge star operator on differential forms, and $\overline{\star} v := \overline{(\star v)}$.

To find more regular solutions of a similar problem for the de Rham complex, the work \cite{OMTNS21} introduced a scale of Bochner-Sobolev functional spaces, one of the modifications of which was later used in solving similar problems for elliptic differential complexes on smooth compact Riemannian manifolds; it is also suitable for the problem \eqref{zad}. More precisely, for $s, k \in \mathbb{Z}_{+}$, let $B_{\mathrm{vel}, q}^{k, 2s, s}$ denote the set of all "velocities", i.e., $(0, q)$-forms $u$ from $C\left(I, \mathbf{H}_{q}^{k+2s}\right) \cap L^{2}\left(I, \mathbf{H}_{q}^{k+1+2s}\right)$ such that

\[
\partial_{x}^{\alpha} \partial_{t}^{j} u \in C\left(I, \mathbf{H}_{q}^{k+2s-|\alpha|-2j}\right) \cap L^{2}\left(I, \mathbf{H}_{q}^{k+1+2s-|\alpha|-2j}\right),
\]

if $|\alpha|+2j \leq 2s$. We will equip the space $B_{\mathrm{vel}, q}^{k, 2s, s}$ with the natural norm

\[
\|u\|_{B_{\mathrm{vel}, q}^{k, s, s}} := \left(\sum_{i=0}^{k} \sum_{|\alpha|+2j \leq 2s}\left\|\partial_{x}^{\alpha} \partial_{t}^{j} u\right\|_{i, q, T}^{2}\right)^{1 / 2},
\]

where $\|u\|_{i, q, T} = \left(\left\|\nabla^{i} u\right\|_{C\left(I, L^2_{\Lambda^{0, q}}\right)}^{2} + \mu\left\|\nabla^{i+1} u\right\|^{2}_{L^{2}\left(I, L^{2}_{\Lambda^{0, q}}\right)}\right)^{1 / 2}$.

Similarly, for $s, k \in \mathbb{Z}_{+}$, let $B_{\mathrm{for}, q}^{k, 2s, s}$ consist of all "external forces", i.e., $(0, q)$-forms $f$ from $C\left(I, H_{q}^{2s+k}\right) \cap L^{2}\left(I, H_{q}^{2s+k+1}\right)$, for which, if $|\alpha|+2j \leq 2s$, it holds that

\[
\partial_{x}^{\alpha} \partial_{t}^{j} f \in C\left(I, H_{q}^{k}\right) \cap L^{2}\left(I, H_{q}^{k+1}\right).
\]

If $f \in B_{\mathrm{for}, q}^{k, 2s, s}$, then indeed

\[
\partial_{x}^{\alpha} \partial_{t}^{j} f \in C\left(I, H_{q}^{k+2(s-j)-|\alpha|}\right) \cap L^{2}\left(I, H_{q}^{k+1+2(s-j)-|\alpha|}\right)
\]

for all $\alpha$ and $j$ satisfying $|\alpha|+2j \leq 2s$. We will equip the space $B_{\mathrm{for}, q}^{k, 2s, s}$ with the natural norm

\[
\|f\|_{B_{\mathrm{for}, q}^{k, 2s, s}} = \left(\sum_{\substack{|\alpha|+2j \leq 2s \\ 0 \leq i \leq k}}\left\|\nabla^{i} \partial_{x}^{\alpha} \partial_{t}^{j} f\right\|_{C\left(I, L_{q}^{2}\right)}^{2} + \left\|\nabla^{i+1} \partial_{x}^{\alpha} \partial_{t}^{j} f\right\|_{L^{2}\left(I, L_{q}^{2}\right)}^{2}\right)^{1 / 2}.
\]

Finally, fix a function $h_{0} \in C_{0}^{\infty}\left(\mathbb{C}^{n}\right)$ such that
\begin{equation}\label{portog.0}
	\int_{\mathbb{R}^{2n}} h_{0}(x) \, dx = 1,
\end{equation}
and define the space $B_{\mathrm{pre}, q-1}^{k+1, 2s, s}$ for "pressure" $p$ as consisting of all $(0, q-1)$-forms from $C\left(I, H_{\mathrm{loc}, q-1}^{2s+2+1}\right) \cap L^{2}\left(I, H_{\mathrm{loc}, q-1}^{2s+k+2}\right)$, satisfying
\begin{equation}\label{portog}
	\int{\mathbb{R}^{2n}} p_{0 J}(x) h_{0}(x) \, dx = 0 \text{ for all } t \in [0, T] \text{ and } \# J = q-1,
\end{equation}
and such that $\overline{\partial}^{q-1} p \in B_{\mathrm{for}, q}^{k, 2s, s}$,
\begin{align}
	\left( \overline{\partial}^{q-2}\right)^{*} p &= 0 \text{ in } \mathbb{C}^{n} \times [0, T], \label{eq1}\\
	\|p\|_{L^{2}\left(I, C_{b, \Lambda^{0, q-1}}\right)} &< +\infty \text{ when } 2s+k = n+1, \label{eq2}\\
	\|p\|_{L^{2}\left(I, C_{b, \Lambda^{0, q-1}}\right)} + \|p\|_{C\left(I, C_{b, \Lambda^{0, q-1}}\right)} &< +\infty \text{ when } 2s+k > n+1; \label{eq3}
\end{align}
here, $C_{b}$ is the space of bounded continuous functions in $\mathbb{C}^{n}$ with the norm $\|w\|_{C_{b}} = \sup_{z \in \mathbb{C}^{n}} |w(z)|$.
This space can be equipped with the norm

 \[
 \|p\|_{B_{\mathrm{pre}, q-1}^{k+1, 2s, s}} = \left\lbrace 
 \begin{array}{ll}
 	\|\overline{\partial}^{q-1} p\|_{B_{\mathrm{for}, q}^{k, 2s, s}}, & 2s + k \leq n, \\
 	\|\overline{\partial}^{q-1} p\|_{B_{\mathrm{for}, q}^{k, 2s, s}} + \|p\|_{L^{2}\left(I, C_{b, \Lambda^{0, q-1}}\right)}, & 2s + k = n + 1, \\
 	\|\overline{\partial}^{q-1} p\|_{B_{\mathrm{for}, q}^{k, 2s, s}} + \|p\|_{L^{2}\left(I, C_{b, \Lambda^{0, q-1}}\right)} + \|p\|_{C\left(I, C_{b, \Lambda^{0, q-1}}\right)}, & 2s + k > n + 1.
 \end{array}
 \right. 
 \]

 Clearly, the spaces $B_{\mathrm{vel}, q}^{k, 2s, s}$, $B_{\mathrm{for}, q}^{k, 2s, s}$, and $B_{\mathrm{pre}, q-1}^{k+1, 2s, s}$ are Banach spaces. Additionally, we will consider the linearization of the problem \eqref{zad}, and for this purpose, we denote
 $$
 \mathbf{B} _q (w, u) = M^{q}_1 (\overline{\partial}^q w, u) + \overline{\partial}^{q-1} M^{q}_2 (w, u) + 
 M^{q}_1 (\overline{\partial}^{q} u, w) + \overline{\partial}^{q-1} M^{q}_2 (u, w),
 $$
 for forms $u$, $w$ of bidegree $(0, q)$.

\begin{lemma}
	\label{l.NS.cont.0}
	Let $n \geq 2$, $s \in \mathbb{N}$, $k \in \mathbb{Z}_+$, $2s + k > n - 1$. Then the mappings
	$$
	\begin{array}{rrcl}
		\overline{\partial}^{q-1} :
		& B^{k+1, 2(s-1), s-1}_{\mathrm{pre}, {q-1}} 
		& \to
		& B^{k, 2(s-1), s-1}_{\mathrm{for}, {q}}, \\
		\Delta :
		& B^{k, 2s, s}_{\mathrm{vel}, {q}}
		& \to
		& B^{k, 2(s-1), s-1}_{\mathrm{for}, {q}}, \\
		\partial_t :
		& B^{k, 2s, s}_{\mathrm{vel}, {q}} 
		& \to
		& B^{k, 2(s-1), s-1}_{\mathrm{for}, {q}}, \\
		\mathcal{N}^q :
		& B^{k, 2s, s}_{\mathrm{vel}, {q}}  
		& \to
		& B^{k, 2(s-1), s-1}_{\mathrm{for}, {q}},
	\end{array}
	$$
	are continuous. Moreover, if $w \in B^{k+2, 2(s-1), s-1}_{\mathrm{vel}, {q}}$, then the mapping
	$$
	\begin{array}{rrcl}
		\mathbf{B}_q (w, \cdot) :
		& B^{k, 2s, s}_{\mathrm{vel}, {q}}  
		& \to
		& B^{k, 2(s-1), s-1}_{\mathrm{for}, {q}},
	\end{array}
	$$
	is also continuous, and for all $u, w \in B^{k+2, 2(s-1), s-1}_{\mathrm{vel}, {q}}$,
	\begin{equation}\label{eq.B.pos.bound}
		\| \mathbf{B}_q(w, u)\|_{B^{k, 2(s-1), s-1}_{\mathrm{for}, {q}}} 
		\leq c^{q}_{s, k} 
		\|w\|_{B^{k+2, 2(s-1), s-1}_{\mathrm{vel}, {q}}} 
		\|u\|_{B^{k+2, 2(s-1), s-1}_{\mathrm{vel}, {q}}},
	\end{equation}
	where $c^{q}_{s, k}$ is a positive constant independent of $u$ and $w$.
\end{lemma}
\begin{proof}
For the linear operators $\overline{\partial}$, $\partial_t$, and $\Delta^q$, the statement of the lemma follows directly from the definitions of the spaces, and for the operators $\mathbf{B}_q (w, \cdot)$ and $\mathcal{N}^q$, it follows from the Gagliardo-Nirenberg inequalities; see, for example, \cite[Lemma 3.5]{OMTNS21} or \cite[Theorem 1.4]{Polk23}.
\end{proof}

Next, let $ \varphi^q $ denote the fundamental solution of the generalized Laplacian operator $ \Delta^{q} $ (see, for example, \cite{CDE95}). Consider the projection
$\mathrm {P}^q$ of the space $ B^{k, 2(s-1), s-1}_{\mathrm{for}, {q}}$ onto the kernel of the operator $ (\overline{\partial}^{q-1})^* $.

\begin{lemma}\label{proector} 
	Let $s, k \in \mathbb{Z}_+$. 
	For each $q$, the pseudodifferential operator $\mathrm {P}^q = \varphi^q(\overline{\partial}^q)^* \overline{\partial}^q$ induces a continuous mapping
	\begin{equation}\label{cont.proetor}
		\mathrm {P}^q: B^{k, 2(s-1), s-1}_{\mathrm{for}, {q}} \to B^{k, 2(s-1), s-1}_{\mathrm{vel}, {q}},
	\end{equation}
	such that
	\begin{equation*}
		\mathrm {P}^q \circ \mathrm {P}^q u = \mathrm {P}^q u,\quad
		(\mathrm {P}^q u, v)_{L^{2}_{q}} = (u, \mathrm {P}^q v)_{L^{2}_{q}},\quad
		(\mathrm {P}^q u, (I - \mathrm {P}^q) u)_{L^{2}_{q}} = 0
	\end{equation*}
	for all $u, v \in C_{0, \Lambda^{0, q}}^{\infty}$.
\end{lemma}
\begin{proof}
	Indeed, let $u, v \in C_{0, \Lambda^{0, q}}^{\infty}$. Since $ \varphi^q $ is the fundamental solution of the generalized Laplacian operator, we have
	\begin{equation}\label{eq22}
		v = \Delta^{q} \varphi^q v = \varphi^q \Delta^{q} v = \varphi^q (\overline{\partial}^{q})^{*} \overline{\partial}^{q} v + \varphi^q \overline{\partial}^{q-1} (\overline{\partial}^{q-1})^{*} v.
	\end{equation}
	Due to the fact that $\overline{\partial}^{q+1} \circ \overline{\partial}^{q} = 0$, we have

	\[
	\mathrm{P}^q \circ \mathrm{P}^q u = ((\overline{\partial}^{q})^* \overline{\partial}^{q} \varphi^q) \circ ((\overline{\partial}^{q})^* \overline{\partial}^{q} \varphi^q) u =
	\]

	\[
	((\overline{\partial}^{q})^* \overline{\partial}^{q} \varphi^q (\overline{\partial}^{q})^* \overline{\partial}^{q} \varphi^q) u = (\overline{\partial}^{q})^* \overline{\partial}^{q} \Delta^q \varphi^q u = \mathrm{P}^q u.
	\]

Next, from the formula \eqref{eq22}, we obtain the equality

\[
\mathrm{P}^q = I - \overline{\partial}^{q-1} (\overline{\partial}^{q-1})^* \varphi^q
\]

for all functions with compact support in $\mathbb{R}^{2n}$. Therefore,

\[
(\mathrm{P}^q u, v)_{L^2_q} = (\mathrm{P}^q u, \mathrm{P}^q v + \overline{\partial}^{q-1} (\overline{\partial}^{q-1})^* \varphi^q v)_{L^2_q} = (\mathrm{P}^q u, \mathrm{P}^q v)_{L^2_q} =
\]

\[
(u - \overline{\partial}^{q-1} (\overline{\partial}^{q-1})^* \varphi^q u, \mathrm{P}^q v)_{L^2_q} = (u, \mathrm{P}^q v)_{L^2_q},
\]

since $(\overline{\partial}^{q-1})^* \mathrm{P}^q = 0$. On the other hand,

\[
(\mathrm{P}^q u, (I - \mathrm{P}^q) u)_{L^2_q} = (\mathrm{P}^q u, u)_{L^2_q} - (\mathrm{P}^q u, \mathrm{P}^q u)_{L^2_q} = 0.
\]

Finally, the continuity of the mapping $\mathrm{P}^q: B^{k, 2(s-1), s-1}_{\mathrm{for}, {q}} \to B^{k, 2(s-1), s-1}_{\mathrm{vel}, {q}}$ follows from Lemma \ref{l.NS.cont.0} and the commutative equality
$\mathrm{P}^q \partial_t^j = \partial_t^j \mathrm{P}^q$ for
$j \leq s-1$.

\end{proof}

\begin{lemma}\label{p.nabla.Bochner} 
	Let $n \geq 3$, $1 \leq q < n$, $2s + k > n$, and the form $F \in B^{k, 2(s-1), s-1}_{\mathrm{for}, q}$ satisfies $\mathrm{P}^q F = 0$ in $\mathbb{C}^n$.  
	Then there exists a unique form
	$p \in B^{k+1, 2(s-1), s-1}_{\mathrm{pre}, {q-1}}$
	such that (\ref{portog}) holds and
	\begin{equation}\label{eq.nabla.Bochner}
		\overline{\partial}^{q-1} p = F \text{ in } \mathbb{C}^n \times [0, T].
	\end{equation}
\end{lemma}
\begin{proof}
	Let the conditions of the lemma be satisfied. We will show that the $(0, q-1)$-form

	\[
	p = (\overline{\partial}^{q-1})^* \varphi^q F
	\]

	is a solution to (\ref{eq.nabla.Bochner}). Indeed,

	\[
	\overline{\partial}^{q-1} p = \overline{\partial}^{q-1} (\overline{\partial}^{q-1})^* \varphi^q F,
	\]

	however, using Lemma \ref{proector} and the equality \eqref{eq22}, we see that for any function $v \in C_{0, \Lambda^{0, q}}^{\infty}$, it holds that

	\[
	\langle p, (\overline{\partial}^{q-1})^* v\rangle_q = \langle (\overline{\partial}^{q-1})^* \varphi^q F, (\overline{\partial}^{q-1})^* v\rangle_q = \langle F, \varphi^q \overline{\partial}^{q-1} (\overline{\partial}^{q-1})^* v\rangle_q =
	\]

	\[
	= \langle F, v\rangle_q - \langle F, \varphi^q (\overline{\partial}^{q})^* \overline{\partial}^{q} v\rangle_q = 
	\langle F, v\rangle_q,
	\]

	since $\mathrm{P}^q F = 0$. By construction of the solution, we have $\overline{\partial}^{q-1} p = F$ and $(\overline{\partial}^{q-2})^* p = 0$.

	Let now $p_1, p_2 \in B^{k+1, 2(s-1), s-1}_{\mathrm{pre}, q-1}$ be two solutions of equation (\ref{eq.nabla.Bochner}). Then $p = p_1 - p_2$ is also a solution, and $\overline{\partial}^{q-1} p = 0$. Since $p \in B^{k+1, 2(s-1), s-1}_{\mathrm{pre}, q-1}$, we have
	$(\overline{\partial}^{q-2})^* p = 0$, which means that $p$ actually has harmonic coefficients in $\mathbb{C}^n$ bounded at infinity. Then, by Liouville's theorem, $p$ is a constant vector, i.e., the "pressure" is determined up to a constant, as in the case of the classical Stokes equations.
	To ensure the uniqueness of the vector $p$, we use the fact that it, by the definition of the space $B^{k+1, 2(s-1), s-1}_{\mathrm{pre}, q-1}$, satisfies \eqref{portog}. Therefore, taking into account \eqref{portog.0},
	we conclude that $p \equiv 0$.
\end{proof}

To obtain the open mapping theorem, we need to consider the linearization of the problem \eqref{zad}.
Namely, given $(0, q)$-forms $f$ and $w$ with sufficiently smooth coefficients in
$\mathbb{C}^n \times [0, T]$ and a $(0, q)$-form 
$u_0$ in $\mathbb{C}^n$, it is required to find sufficiently smooth $(0, q)$-form 
$u$ and $(0, q-1)$-form $p$
in the strip $\mathbb{C}^n \times [0, T]$, satisfying
\begin{equation}\label{zadlin}
	\left\{\begin{array}{rcl}
		\partial_{t} u + \mu \Delta^{q} u + \mathbf{B}_q(w, u) + \overline{\partial}^{q-1} p &=& f \text{ in } \mathbb{C}^{n} \times (0, T), \\
		\left(\overline{\partial}^{q-1}\right)^{*} u &=& 0 \text{ in } \mathbb{C}^{n} \times (0, T), \\
		\left(\overline{\partial}^{q-2}\right)^{*} p &=& 0 \text{ in } \mathbb{C}^{n} \times (0, T), \\
		u(z, 0) &=& u_{0}(z), \, z \in \mathbb{C}^{n}, \\
		\sup\limits_{t \in [0, T]} \int\limits_{\mathbb{C}^{n}} |u(z, t)|^{2} \, dx + \int\limits_{0}^{T} \int\limits_{\mathbb{C}^{n}} \sum\limits_{j=1}^{2n} \left|\partial_{j} u(z, t)\right|^{2} \, dx \, dt &<& +\infty.
	\end{array}\right.
\end{equation}

\begin{thm}
	\label{t.exist.NS.lin.strong}
	Let $n \geq 2$, $0 \leq q < n$, 
	$s \in \mathbb{N}$, $k \in \mathbb{Z}_+$, $2s + k > n$, 
	and
	$w \in B^{k, 2s, s}_{\mathrm{vel}, q}$. 
	Then the problem \eqref{zadlin} induces a bijective continuous linear mapping
	\begin{equation}
		\label{eq.map.Aw}
		\mathcal{A}^{q}_w :
		B^{k, 2s, s}_{\mathrm{vel}, q}  \times 
		B^{k+1, 2(s-1), s-1}_{\mathrm{pre}, {q-1}} 
		\to B^{k, 2(s-1), s-1}_{\mathrm{for}, q}  \times 
		\mathbf{H}^{2s+k}_{q},
	\end{equation}
	with a continuous inverse operator $(\mathcal{A}^{q}_w)^{-1}$.
\end{thm}

\begin{proof}
The statement follows from Lemma \ref{p.nabla.Bochner}, which allows "reconstruction" of the pressure $p$ after applying the projection operator ${P}^q$ to the problem \eqref{zadlin}, using the standard Galerkin method (see, for example, \cite{OMTNS21}).
\end{proof}

Since $B^{k, 2s, s}_{\mathrm{vel}, q}$ is continuously embedded in the space, the uniqueness theorem also holds for the nonlinear case on the scales of the spaces we have introduced.

However, in recent years, the scientific community's efforts have also been directed towards seeking proof of the absence of an existence theorem for Navier-Stokes type equations in high spatial dimensions; see, for example, \cite{B24, Lad03, PlSv03, Tao16, Temam79}. Nevertheless, we obtain the open mapping theorem or, in other words, the stability theorem, for the problem \eqref{zad} on the introduced scale of Bochner-Sobolev spaces.

\begin{thm}\label{theorOpen}
	Let $n \geq 2$, $1 \leq q < n$, $s \in \mathbb{N}$, and $k \in \mathbb{Z}_+$, $2s + k > n$. Then \eqref{zad} induces an injective continuous nonlinear open mapping
\begin{equation}\label{eq.map.A}
	\mathcal{A}^{q}: B_{\mathrm{vel}, q}^{k, 2 s, s} \times B_{\mathrm{pre}, q-1}^{k+1,2(s-1), s-1} \rightarrow B_{\mathrm{for}, q}^{k, 2(s-1), s-1} \times \mathbf{H}_{q}^{2 s+k}.
\end{equation}
\end{thm}

\begin{proof}
	The continuity of the operator $\mathcal{A}^{q}$ follows from Lemma \ref{l.NS.cont.0}.
	Next, let 
	$$
	\begin{array}{rcl}
		(u, p)
		& \in
		& B^{k, 2s, s}_{\mathrm{vel}, q} \times 
		B^{k+1, 2(s-1), s-1}_{\mathrm{pre}, {q-1}},
		\\
		\mathcal{A}^{q} (u, p)
		\, = \, (f, u_0)
		& \in
		& B^{k, 2(s-1), s-1}_{\mathrm{for}, q} \times 
		\mathbf{H}^{2s+k}_{q}.
	\end{array}
	$$

	Thus, $u$ is a weak solution of the problem \eqref{zad}, i.e., it satisfies \eqref{zadsl}. We will show that the problem \eqref{zad} has at most one solution $(u, p)$ in the space $B^{k, 2s, s}_{\mathrm{vel}, q} \times 
	B^{k+1, 2(s-1), s-1}_{\mathrm{pre}, q-1}$. Indeed, let $(u', p')$ and $(u'', p'')$ be two solutions of the problem \eqref{zad} in the given functional spaces, i.e., $\mathcal{A}^{q} (u', p') = \mathcal{A}^{q} (u'', p'')$. Then the forms $u = u' - u''$ and $p = p' - p''$ satisfy \eqref{zad} with zero data $(f, u_0) = (0, 0)$, hence
	\begin{equation*}
		\frac{d}{dt} \|u\|^2_{{L}_{q}^2}
		+ 2\mu \|\overline{\partial}^{q} u\|^2_{{L}_{q+1}^2} = \Big(\left( \mathbf{B}_q (u'', u'') - \mathbf{B}_q (u', u')\right), u\Big)_{{L}_{q}^2}.
	\end{equation*}
	From G?rding's inequality and Gr?nwall's lemma (see, for example, \cite{MPF91}), it follows that $u \equiv 0$, and from Lemma \ref{p.nabla.Bochner} we have $p' = p''$. Thus, we have proved the injectivity of the operator $\mathcal{A}^{q}$.

	Finally, it is easy to see that the Fr?chet derivative
	$(\mathcal{A}^{q}_{(w, p_0)})'$
	of the nonlinear mapping $\mathcal{A}^{q}$ at an arbitrary point
	$$
	(w, p_0)
	\in B^{k, 2s, s}_{\mathrm{vel}, q} \times 
	B^{k+1, 2(s-1), s-1}_{\mathrm{pre}, {q-1}},
	$$
	is equal to the continuous linear mapping $\mathcal{A}^{q}_w$. By Theorem \ref{t.exist.NS.lin.strong}, the operator

	\[
	\mathcal{A}^{q}_w:
	B^{k, 2s, s}_{\mathrm{vel}, q} \times 
	B^{k+1, 2(s-1), s-1}_{\mathrm{pre}, {q-1}} \to
	B^{k, 2(s-1), s-1}_{\mathrm{for}, q} \times 
	\mathbf{H}^{2s+k}_{q}
	\]

	is continuously invertible. Thus, the openness of the image of the mapping $\mathcal{A}^{q}$ and the continuity of its local inverse mapping follow from the implicit function theorem for Banach spaces (see, for example, \cite[Theorem 5.2.3]{Ham82}).
	
\end{proof}

In particular, the theorem means that for any pair of data for which a solution in the desired class exists, there is a neighborhood for all elements of which corresponding solutions also exist. Note that for the classical Navier-Stokes equations in other functional spaces, a similar statement was noted in the book by O. A. Ladyzhenskaya \cite{Lad61}. For some other elliptic complexes, similar theorems in various functional spaces were obtained in \cite{Polk23}, \cite{OMTNS21}.

Furthermore, from Theorem \ref{theorOpen}, it follows that the image of the operator \eqref{eq.map.A} is closed if and only if it coincides with the entire space $ B_{\mathrm{for}, q}^{k, 2(s-1), s-1} \times \mathbf{H}_{q}^{2s+k}$.

Recently, a surjectivity criterion for the image in spaces similar to those introduced by us for the Navier-Stokes equations was obtained, see \cite{VUUM22}, inspired by considerations from \cite{Sm65}. For our problem, we can obtain a similar criterion.

\begin{thm}
	\label{t.LsLr}
	Let $s \in \mathbb{N}$,
	$k \in \mathbb{Z}_+$,
	and the numbers $\mathfrak{r}$, $\mathfrak{s}$ satisfy 
	$2/\mathfrak{s} + 2n/\mathfrak{r} = 1$.
	Then the mapping \eqref{eq.map.A} is surjective if and only if the precompactness of the image $\mathcal{A}^{q}(S)$ in the space $B^{k, 2(s-1), s-1}_{\mathrm{for}, q} \times \mathbf{H}^{2s+k}_{q}$ of any subset $S = S_{\mathrm{vel}, q} \times S_{\mathrm{pre}, q-1}$ of the Cartesian product 
	$B^{k, 2s, s}_{\mathrm{vel}, q} \times B^{k+1, 2(s-1), s-1}_{\mathrm{pre}, q-1}$ implies the boundedness of the set $S_{\mathrm{vel}, q}$ in the space $L^{\mathfrak{s}}(I, L_{q}^{\mathfrak{r}})$.
\end{thm}
\begin{proof}
	The proof is similar to the case when the Navier-Stokes equation associated with the de Rham complex is considered, see \cite[Theorem 3]{VUUM22}. We present the main steps of the proof. Let the mapping \eqref{eq.map.A} be surjective. 
	Then the image of this mapping is closed by Theorem \ref{theorOpen}. Fix a 
	subset $S = S_{\mathrm{vel}, q} \times S_{\mathrm{pre}, q-1}$ of the product
	$B^{k, 2s, s}_{\mathrm{vel}, q} \times B^{k+1, 2(s-1), s-1}_{\mathrm{pre}, q-1}$ 
	such that the image $\mathcal{A}^{q}(S)$ is precompact in the space
	$B^{k, 2(s-1), s-1}_{\mathrm{for}, q} \times \mathbf{H}^{2s+k}_{q}$. 
	If the set $S_{\mathrm{vel}, q}$ is unbounded in the space
	$L^{\mathfrak{s}}(I, {L}^{\mathfrak{r}}_{q})$, then 
	there exists a sequence $\{ (u_k, p_k) \} \subset S$ such that
	\begin{equation}
		\label{eq.unbounded}
		\lim_{k \to \infty} \| u_k \|_{L^{\mathfrak{s}}(I, {L}^{\mathfrak{r}}_{q})}
		= \infty.
	\end{equation}
	Since the set $\mathcal{A}^{q}(S)$ is precompact in
	$B^{k, 2(s-1), s-1}_{\mathrm{for}, q} \times \mathbf{H}^{2s+k}_{q}$,
	we conclude that the corresponding sequence of data
	$\{ \mathcal{A}^{q} (u_k, p_k) = (f_k, u_{k,0}) \}$
	contains a subsequence $\{ (f_{k_m}, u_{k_m, 0}) \}$ that converges to an element
	$(f, u_0)$ in this space. But the image of the operator $\mathcal{A}^{q}$ is closed, which means that 
	for the data $(f, u_0)$ there exists a unique solution $(u, p)$ for 
	\eqref{zad} 
	in the space
	$B^{k, 2s, s}_{\mathrm{vel}, q} \times B^{k+1, 2(s-1), s-1}_{\mathrm{pre}, q-1}$,
	and the sequence $\{ (u_{k_m}, p_{k_m}) \}$ converges to $(u, p)$ in this space.
	Therefore, $\{ (u_{k_m}, p_{k_m}) \}$ is bounded in
	$B^{k, 2s, s}_{\mathrm{vel}, q} \times B^{k+1, 2(s-1), s-1}_{\mathrm{pre}, q-1}$,
	and this contradicts \eqref{eq.unbounded}, since the space $B^{k, 2s, s}_{\mathrm{vel}, q}$
	is continuously embedded in the space $L^{\mathfrak{s}}(I, {L}^{\mathfrak{r}}_{q})$ for any pair $\mathfrak{r}$, $\mathfrak{s}$ satisfying the conditions 
	$2/\mathfrak{s} + 2n/\mathfrak{r} = 1$.
	
	Next, let us consider the standard a priori estimates for the equations \eqref{zad}. 
	Let the elements of the considered spaces possess sufficient regularity; the estimates are needed to prove the 
	surjectivity of the mapping \eqref{eq.map.A}, rather than to improve 
	the regularity of weak solutions.
	
	\begin{lemma}
		\label{p.En.Est.u.strong}
		If
		$(u, p) \in B^{0, 2, 1}_{{\mathrm{vel}, q}} \times B^{1, 0, 0}_{{\mathrm{pre}, q-1}}$
		is a solution of the equations \eqref{zad} with the data
		$(f, u_0) \in B^{0, 0, 0}_{{\mathrm{for}, q}} \times \mathbf{H}^{2}_{q}$,
		then
		$ \| u \|_{0, \mu, T}
		\leq
		\| (f, u_0) \|_{0, \mu, T}$.
	\end{lemma}
	
	\begin{proof} It follows from standard a priori estimates.
	\end{proof}

Absolutely, the next step involves estimating the derivatives of $u$ and $p$ with respect to $x$ and $t$.

\begin{lemma}
	\label{c.En.Est.g.ks}
Suppose that 
$s \in \mathbb{N}$,
$k \in \mathbb{Z}_+$,
and 
$\mathfrak{s}$, $\mathfrak{r}$ satisfy $2/\mathfrak{s} + 2n/\mathfrak{r} = 1$. 
If 
$(u, p) \in B^{k, 2s, s}_{\mathrm{vel}, q} \times B^{k+1, 2(s-1), s-1}_{\mathrm{pre}, q-1}$
is a solution of equations \eqref{zad} 
with the data 
$(f, u_0) \in B^{k, 2(s-1), s-1}_{\mathrm{for}, q} \times \mathbf{H}^{2s+k}_{q}$,
then it satisfies the estimate
\begin{equation}
	\label{eq.En.Est.Bks}
	\| (u, p) \|_{B^{k, 2s, s}_{\mathrm{vel}, q} \times B^{k+1, 2(s-1), s-1}_{\mathrm{pre}, q-1}}
	\leq
	c(k, s, (f, u_0), u),
\end{equation}
where the constant on the right-hand side depends on
$\| f \|_{B^{k, 2(s-1), s-1}_{\mathrm{for}, q}}$,
$\| u_0 \|_{\mathbf{H}^{2s+k}_{q}}$,
and
$\| u \|_{L^\mathfrak{s}(I, {L}^\mathfrak{r})}$,
as well as on
$\mathfrak{r}$, $T$, $\mu$, etc.
\end{lemma}
	
	\begin{proof}
	Indeed, it follows from H?lder's and Gagliardo-Nirenberg inequalities.
	\end{proof}

Now we need to show that the image of the mapping \eqref{eq.map.A} is closed if the given subset $S = S_{\mathrm{vel}, q} \times 
S_{\mathrm{pre}, q-1}$ of the Cartesian product 
$B^{k, 2s, s}_{\mathrm{vel}, q} \times B^{k+1, 2(s-1), s-1}_{\mathrm{pre}, q-1}$
is such that the image $\mathcal{A}^{q}(S)$ is precompact in the space
$B^{k, 2(s-1), s-1}_{\mathrm{for}, q} \times \mathbf{H}^{2s+k}_{q}$,
then the set $S_{\mathrm{vel}, q}$ is bounded in the space
$L^{\mathfrak{s}}(I, {L}^{\mathfrak{r}}_{q})$ with the pair $\mathfrak{s}$, $\mathfrak{r}$ satisfying 
$2/\mathfrak{s} + 2n/\mathfrak{r} = 1$.

Let the pair
$(f, u_0) \in B^{k, 2(s-1), s-1}_{\mathrm{for}, q} \times \mathbf{H}^{2s+k}_{q}$
belong to the closure of the image of the operator $\mathcal{A}^{q}$.
Then there exists a sequence $\{ (u_i, p_i) \}$ in
$B^{k, 2s, s}_{\mathrm{vel}, q} \times B^{k+1, 2(s-1), s-1}_{\mathrm{pre}, q-1}$
such that the sequence
$\{ (f_i, u_{i,0}) = \mathcal{A}^{q}(u_i, p_i) \}$
converges to $(f, u_0)$ in the space $B^{k, 2(s-1), s-1}_{\mathrm{for}, q} \times \mathbf{H}^{2s+k}_{q}$.

Consider the set $S = \{ (u_i, p_i) \}$.
Since the image $\mathcal{A}^{q}(S) = \{ (f_i, u_{i,0}) \}$ is precompact in
$B^{k, 2(s-1), s-1}_{\mathrm{for}, q} \times \mathbf{H}^{2s+k}_{q}$,
it follows from our assumption that the subset
$S_{\mathrm{vel}, q} = \{ u_i \}$ of
$B^{k, 2s, s}_{\mathrm{vel}, q}$
is bounded in the space 
$L^\mathfrak{s}(I, {L}^{\mathfrak{r}}_{q})$. 

By applying Lemmas \ref{p.En.Est.u.strong} and \ref{c.En.Est.g.ks}, we conclude that the sequence
$\{ (u_i, p_i) \}$
is bounded in the space
$ B^{k, 2s, s}_{\mathrm{vel}, q} \times B^{k+1, 2(s-1), s-1}_{\mathrm{pre}, q-1} $.
By definition of $B^{k, 2s, s}_{\mathrm{vel}, q}$, the sequence $\{ u_i \}$ is bounded in
$C (I, \mathbf{H}^{k+2s}_{q})$ and
$L^2 (I, \mathbf{H}^{k+2s+1}_{q})$,
and the partial derivatives $\{ \partial_t^j u_i \}$ with respect to time for $1 \leq j \leq s$
are bounded in
$C (I, \mathbf{H}^{k+2(s-j)}_{q})$ and
$L^2 (I, \mathbf{H}^{k+2(s-j+1)}_{q})$.
Therefore, there exists a subsequence $\{ u_{i_k} \}$ such that:
\begin{enumerate}
	\item The sequence
	$\{ \partial^{\alpha+\beta}_x \partial_t^j u_{i_k} \}$
	converges weakly in $L^2 (I, {L}^{2}_{q})$ provided that
	$|\alpha| + 2j \leq 2s$ and
	$|\beta| \leq k+1$;
	\item The sequence
	$\{ \partial^{\alpha+\beta}_x \partial_t^j u_{i_k} \}$
	converges weakly$^*$ in $L^\infty (I, {L}^{2}_{q})$ provided that
	$|\alpha| + 2j \leq 2s$ and
	$|\beta| \leq k$.   
\end{enumerate}
Clearly, the limit $u$ of the sequence $\{ u_{i_k} \}$ is a solution to 
the equations \eqref{zad}, such that
\begin{enumerate}
	\item Each derivative $\partial^{\alpha+\beta}_x \partial_t^j u$ 
	belongs to $L^2 (I, \mathbf{H}^{0}_{q})$ provided that 
	$|\alpha| + 2j \leq 2s$ and $|\beta| \leq k+1$; 
	\item Each derivative $\partial^{\alpha+\beta}_x \partial_t^j u$
	belongs to $L^\infty (I, \mathbf{H}^{0}_{q})$ provided that
	$|\alpha| + 2j \leq 2s$ and $|\beta| \leq k$. 
\end{enumerate}
From the energy estimates and Gr?nwall's lemma, it follows that such a strong solution will be unique (see, for example, \cite{Temam79} for the Navier-Stokes equations).
Moreover, if 
\begin{equation}
	\label{eq.0js-1}
	0 \, \leq \, j
	\leq
	s-1,
	\quad
	|\alpha| + 2j
	\leq
	2s,
	\quad
	|\beta|
	\leq
	k,
\end{equation}
then
$\partial^{\alpha+\beta}_x \partial_t^j u \in L^2 (I, \mathbf{H}^{1}_{q})$ and
$\partial^{\alpha+\beta}_x \partial_t^{j+1} u \in L^2 (I, (\mathbf{H}^{1}_{q})')$.
It follows that $\partial^{\alpha+\beta}_x \partial_t^j u \in C (I, \mathbf{H}^{0}_{q})$
for all $j$, $\alpha$, and $\beta$ satisfying (\ref{eq.0js-1}). 
Hence, $u$ belongs to the space
$B^{k+2, 2(s-1), s-1}_{\mathrm{vel}, q}$.
Moreover, using formula \eqref{eq.B.pos.bound}
for $w = u$, we obtain that the derivatives
$\partial^{\alpha+\beta}_x \partial_t^j \mathcal{N}^q u$
belong to $C (I, {L}^2_{q})$ for all $j$, $\alpha$, $\beta$ satisfying
inequalities (\ref{eq.0js-1}).

Furthermore, the operator ${P}^{q}$ maps
$C (I, {L}^2_{q})$ continuously into
$C (I, {L}^2_{q})$.
Therefore, since $u$ is a solution to \eqref{zad}, we have
$$
\partial^{\beta}_x \partial_t^s u
= \partial^{\beta}_x \partial_t^{s-1} \mu \varDelta u
- \partial^{\beta}_x \partial_t^{s-1} {P}^q \mathcal{N}^q u
+ \partial^{\beta}_x \partial_t^{s-1} {P}^q f
$$
which belongs to $C (I, \mathbf{H}^{0}_{q})$ for all multi-indices $\beta$ such that $|\beta| \leq k$.
In other words, $u \in B^{k, 2s, s}_{\mathrm{vel}, q}$.
Finally, from Lemma \ref{p.nabla.Bochner} it follows that there exists $p \in B^{k+1, 2(s-1), s-1}_{\mathrm{pre}, q-1}$
such that 
$$
\overline{\partial}^{q-1} p
= (I - {P}^q) (f - \mathcal{N}^q u),
$$
i.e., the pair $(u, p) \in B^{k, 2s, s}_{\mathrm{vel}, q} 
\times B^{k+1, 2(s-1), s-1}_{\mathrm{pre}, q-1}$ is a solution of \eqref{zad}.

Thus, we have proved that the image of the mapping in \eqref{eq.map.A} is closed.
It follows that the mapping \eqref{eq.map.A} is surjective.

\end{proof}

\bigskip

\textit{This research is supported by the Krasnoyarsk Mathematical Center and funded by the Ministry of Science and Higher Education of the Russian Federation (Agreement No. 075-02-2024-1429).}


\begin{thebibliography}{00}
	
	
	\bibitem{Lad61}
	 Ladyzhenskaya O.~A., 
	 Mathematical Problems of Incompressible Viscous Fluid,
	{\it Nauka}, Moscow, 1970. 
	
	\bibitem{Lad03}
	 Ladyzhenskaya O.~A., 
	The sixth prize millenium problem: Navier-Stokes equations, existence and 
		smoothness,  {\it Russian Math. Surveys}, \textbf{58}:2 (2003),  45--78.
	
	\bibitem{BS17}
	Barker~T., Seregin~G.
	A necessary condition of potential blowup for the Navier-Stokes system in half-space.
	\textit{Mathematische Annalen}, 2017, vol.~369, no.~3-4, pp.~1327--1352.
	
	\bibitem{BC24}
	Barker~T., Prange~C.
	From Concentration to Quantitative Regularity: A Short Survey of Recent Developments for the Navier-Stokes Equations.
	\textit{Vietnam Journal of Mathematics}, 2024, vol.~52, pp.~707--734.
	
	\bibitem{B24}
	Barker~T.
	Higher integrability and the number of singular points for the Navier-Stokes equations with a scale-invariant bound.
	\textit{Proc. Amer. Math. Soc. Ser. B}, 2024, vol.~11, pp.~436-451.
	
	\bibitem{BSS18}
	Barker~T., Seregin~G., $\rm \check{S}$ver\'ak~V.
	On stability of weak Navier-Stokes solutions with large $L^{3,\infty}$ initial data.
	\textit{Communications in Partial Differential Equations}, 2018, vol.~43, no.~4, pp.~628--651.
	
	\bibitem{CWY18}
	Choe~H.~J., Wolf~J., Yang~M.
	A new local regularity criterion for suitable weak solutions of the Navier-Stokes equations in terms of the velocity gradient.
	\textit{Mathematische Annalen}, 2018, vol.~370, no.~3-4, pp.~629--647.
	
	\bibitem{ESS}
	Escauriaza~L., Seregin~G.~A., $\rm \check{S}$ver\'ak~V.
	$L^{3,\infty}$-solutions of the Navier-Stokes equations and backward uniqueness.
	\textit{Russian Mathematical Surveys}, 2003, vol.~58, no.~2, pp.~211--250.
	
	\bibitem{Ham82}
	Hamilton~R.~S.
	The inverse function theorem of Nash and Moser.
	\textit{Bull. of the AMS}, 1982, vol.~7, no.~1, pp.~65--222.
	
	\bibitem{LiMa72}
	Lions~J.~L., Magenes~E.
	\textit{Non-Homogeneous Boundary Value Problems and Applications.}
	Berlin et al: Springer-Verlag, 1972.
	
	\bibitem{NSE19}
	Mera~A., Tarkhanov~N., Shlapunov~A.~A.
	Navier-Stokes Equations for Elliptic Complexes.
	\textit{Journal of Siberian Federal University, Math. and Phys.}, 2019, vol.~12, no.~9, pp.~3--27.
	
	\bibitem{MPF91}
	Mitrinovi\'c D.~S., Pe$\check{c}$ari\'c J.~E, Fink A.~M.
	\textit{Inequalities Involving Functions and Their Integrals and Derivatives.}
	Mathematics and its Applications (East European Series), V.~53.
	Dordrecht: Kluwer Academic Publishers, 1991.
	
	\bibitem{PlSv03}
	Plech\'ac~P., $\rm \check{S}$ver\'ak~V.
	Singular and regular solutions of a nonlinear parabolic system.
	\textit{Nonlinearity}, 2003, vol.~16, no.~6, pp.~2083--2097.
	
	\bibitem{Polk23}
	Polkovnikov~A.~N.
	An open mapping theorem for nonlinear operator equations associated with elliptic complexes.
	\textit{Applicable Analysis}, 2023, vol.~102, pp.~2211--2233.
	
	\bibitem{Pro59}
	Prodi~G.
	Un teorema di unicit\'a per le equazioni di Navier-Stokes.
	\textit{Annali di Matematica Pura ed Applicata}, 1959, vol.~48, pp.~173--182.
	
	\bibitem{Serr62}
	Serrin~J.
	On the interior regularity of weak solutions of the Navier-Stokes equations.
	\textit{Archive for Rational Mechanics and Analysis}, 1962, vol.~9, pp.~187--195.
	
	\bibitem{Sm65}
	Smale~S.
	An infinite dimensional version of Sard's theorem.
	\textit{Amer. J. Math.}, 1965, vol.~87, no.~4, pp.~861--866.
	
	\bibitem{OMTNS21}
	Shlapunov~A.~A., Tarkhanov~N.
	An open mapping theorem for the Navier-Stokes type equations associated with the de Rham complex over $\mathbb{R}^{n}$.
	\textit{Siberian Electronic Math. Reports}, 2021, vol.~18, no.~2, pp.~1433--1466.
	
	\bibitem{VUUM22}
	Shlapunov~A.~A., Tarkhanov~N.
	Inverse image of precompact sets and regular solutions to the Navier-Stokes equations.
	\textit{Vestn. Udmurtsk. Univ. Mat. Mekh. Komp. Nauki}, 2022, vol.~32, no.~2, pp.~278--297.
	
	\bibitem{Tao16}
	Tao~T.
	Finite time blow-up for an averaged three-dimensional Navier-Stokes equation.
	\textit{J. of the AMS}, 2016, vol.~29, pp.~601--674.
	
	\bibitem{CDE95}
	Tarkhanov~N.
	\textit{Complexes of differential operators.}
	Dordrecht, NL: Kluwer Academic Publishers, 1995.
	
	\bibitem{Temam79}
	Temam~R.
	\textit{Navier-Stokes Equations. Theory and Numerical Analysis.}
	Amsterdam: North Holland Publ. Comp., 1979.
	
\end{thebibliography}
\end{document}